\documentclass[12pt]{amsart} 
\usepackage[top=1in, bottom=1in, left=1in, right=1in]{geometry}

\usepackage{times}
\usepackage[table,dvipsnames]{xcolor}
\usepackage{graphicx}

\PassOptionsToPackage{hyphens}{url}\usepackage[colorlinks=true, linkcolor=NavyBlue, citecolor=teal, urlcolor=gray]{hyperref}   
\numberwithin{equation}{section}

\usepackage{scrextend}
\usepackage{amssymb,mathrsfs,amsmath}
\usepackage{mathtools}
\usepackage{bbm}

\theoremstyle{plain}
\newtheorem{thm}{Theorem}
\newtheorem{prop}{Proposition}[section]
\newtheorem{lem}[prop]{Lemma}

\newtheorem{conj}{Conjecture}

\theoremstyle{definition}

\theoremstyle{remark}
\newtheorem*{rem*}{Remark}
\newtheorem*{rems*}{Remarks}

\newcommand{\Z}{\mathbb{Z}}

\newcommand{\CM}{\mathcal{M}}

\let\lcm\relax
\DeclareMathOperator{\lcm}{lcm}

\newcommand{\bv}\boldsymbol{}

\newcommand{\eps}\varepsilon

\renewcommand{\leq}{\leqslant}

\renewcommand{\geq}{\geqslant}

\renewcommand{\mod}[1]{\,(\mathrm{mod}\,#1)}

\title{On a conjecture of Krukenberg and a problem of Dalton and Trifonov}
\author{Jonah Klein}
\begin{document}
\maketitle
\begin{abstract}
We prove that if the smallest modulus of a covering system with distinct moduli is $5$, then the largest modulus is at least 108. We also prove that if the smallest modulus of a covering system with distinct moduli is $5$, then the least common multiple of the moduli is at least 1440. Finally, we prove that if the smallest modulus of a covering system with distinct moduli is 6, then the least common multiple of the moduli is at least $5040$. The constants $108$, $1440$ and $5040$ are best possible. This resolves a conjecture of Krukenberg, a problem of Dalton and Trifonov, and a generalization thereof. 
\end{abstract}
\section{Introduction}

The \textit{arithmetic progression} $a \mod m$ is the set $\{a+km: k\in \Z\}$. A \textit{covering system} is a finite set of arithmetic progressions, with the property that every integer belongs to at least one of them. Covering systems were introduced by Erd\H{o}s in 1950 \cite{ErdosIntro}. In the years since, they have received a considerable amount of attention, see for example a survey of Porubsk\'y and Sch\"onheim \cite{Porubsky}, or again a more recent survey by Balister \cite{BalisterSurvey}. 

We say that a covering system $\{a_1 \mod{m_1},\ldots,a_n \mod {m_n}\}$ is \textit{distinct} if $m_i \neq m_j$ whenever $i \neq j$. For example, it is not too hard to verify that
\[\{0 \mod{2}, 0 \mod{3}, 1 \mod4, 1 \mod 6, 11\mod{12}\}\]
is a distinct covering system. 

In 1971, Krukenberg \cite{Krukenberg} proved that this covering system is the simplest distinct covering system with minimum modulus 2, in the sense that if the smallest modulus of a distinct covering system is 2, then the largest modulus is at least 12. Krukenberg also showed that if the smallest modulus of a distinct covering system is 3, then the largest modulus is at least 36, and gave an example of a distinct covering system with minimum modulus 3 and largest modulus 36. He also stated without proof that if the smallest modulus of a distinct covering system is 4, then the largest modulus is at least 60, and conjectured that if the smallest modulus is 5, then the largest modulus is 108. He gave examples of covering systems that demonstrate that if these bounds are correct, then they are best possible. 

In 2022, Dalton and Trifonov \cite{Dalton} supplied a proof of Krukenberg's claim about the minimum modulus 4, and reaffirmed Krukenberg's conjecture for the minimum modulus 5. We resolve this conjecture. 

\begin{thm}\label{5108}
If the smallest modulus of a covering system with distinct moduli is 5, then the largest modulus is at least 108. 
\end{thm}
To prove Theorem~\ref{5108}, we use some ideas of Krukenberg \cite{Krukenberg}, some ideas of Dalton and Trifonov \cite{Dalton}, some ideas of the author \cite{Memoire}, and a recent idea of McNew and Setty \cite{McNew}.

 Following Theorem~\ref{5108}, it is natural to ask a similar question where the smallest modulus is 6. Towards this goal, we prove the following theorem. 

\begin{thm}\label{6192}
There is a distinct covering system with minimum modulus 6 and maximum modulus 168. 
\end{thm}
We further conjecture that this is in fact the best we can do. 

\begin{conj}
If the smallest modulus of a covering system with distinct moduli is 6, then the largest modulus is at least 168.
\end{conj}

Krukenberg also constructed distinct covering systems with least modulus $m$, while trying to keep the least common multiple of the moduli $L$ as small as possible. When $m=3$, he achieved $L=120$. When $m=4$, he achieved $L=360$. When $m=5$, he achieved $L=1440$. When $m=6$, he achieved $L=5040$. When $m=7$, he achieved $L=15120$. Dalton and Trifonov proved that these are best possible for $m=3,4$. They also asked whether this is the case for $m=5$. Using similar ideas as for the proof of Theorem~\ref{5108}, we prove that this is the case for $m=5$ and $m=6$ as well. 

\begin{thm}\label{51440}
If the smallest modulus of a covering system with distinct moduli is 5, then the least common multiple of the moduli is at least $1440$. 
\end{thm}

\begin{thm}\label{65040}
If the smallest modulus of a covering system with distinct moduli is 6, then the least common multiple of the moduli is at least $5040$. 
\end{thm}
We also conjecture that the $m=7$ case is best possible. 

\begin{conj}
If the smallest modulus of a covering system with distinct moduli is 7, then the least common multiple of the moduli is at least $15120$. 
\end{conj}

The paper is organized as follows. In Section \ref{lemmas}, we go over some preliminary lemmas. In Section \ref{integerprogramming}, we provide details on the idea of McNew and Setty mentioned above. In Section \ref{5108proof}, we prove Theorem~\ref{5108}. In Section \ref{51440proof}, we prove Theorem~\ref{51440}. In Section \ref{65040proof}, we prove Theorem~\ref{65040}. In Section \ref{6192proof}, we prove Theorem~\ref{6192}.

\section{Preliminary lemmas}\label{lemmas}

In this section, we state a few lemmas that will be useful in the proofs of Theorem~\ref{5108}, Theorem~\ref{51440}, and Theorem~\ref{65040}. When possible, we do not prove the lemmas and instead refer to their proofs in the literature. The first two lemmas we look at are fundamental in the work of Krukenberg \cite{Krukenberg} and Dalton and Trifonov \cite{Dalton}. 

\begin{lem}[\cite{Krukenberg}, Corollary 2.3; \cite{Dalton}, Corollary 8]\label{Krukenberglemma1}
Suppose $C$ is a distinct covering system with all moduli $\leq B$. Let $p$ be a prime and $a$ be a positive integer. If $p^a(p+1)>B$, we may discard all arithmetic progressions with modulus divisible by $p^a$ from $C$ and still have a covering system. 
\end{lem}

\begin{lem}[\cite{Dalton}, Corollary 9]\label{Krukenberglemma2}
Let $C$ be a covering system, and suppose the least common multiple of the moduli in $C$ is $L$. Let $C_{p^a}$ be the subset of $C$ containing all arithmetic progressions with modulus divisible by $p^a$, and suppose that $|C_{p^a}|=p$. Suppose further that the moduli in $C_{p^a}$ are $p^am_1,\ldots,p^am_p$. Then one can replace the arithmetic progression in $C_{p^a}$ by a single arithmetic progression with modulus $p^{a-1}\lcm(m_1,\ldots,m_p)$, and the resulting set of congruences will still be a covering system. 
\end{lem}

The next result we will be using concerns translating a covering system. For a proof, it follows from Lemma 2.2 in \cite{Filaseta}. 

\begin{lem}\label{translation}
Let 
\[C=\{a_1 \mod{m_1}, \ldots, a_n \mod{m_n}\}\]
be a covering system, and let $t \in \Z$. Then
\[C+t:=\{a_1+t \mod{m_1},\ldots,a_n+t \mod{m_n}\}\]
is also a covering system. 
\end{lem}

The next lemma we look at comes from the author's Master's thesis \cite{Memoire}. 

\begin{lem}[\cite{Memoire}, Lemma 2.13]\label{pauto}
Let $C$ be a covering system, $L$ be the least common multiple of the moduli in $C$, and $p$ be a prime dividing $L$. Let $0 \leq a_1<a_2 \leq p-1$ be integers, and for $\alpha \in \{0,1,\ldots,p-1\}$, let 
\[C_p(\alpha):=\{a \mod{m} \in C: p|m, a \equiv  \alpha \mod p\}.\]
Let $t$ be the unique integer in $[0,L)$ such that $t \equiv  (a_2-a_1) \mod{p^{\nu_p(L)}}$ and $t \equiv 0 \mod{q^{\nu_q(L)}}$ for each prime $q|L$, $q \neq p$. Then

\[C':=\bigg(C \setminus(C_p(a_1) \cup C_p(a_2))\bigg) \cup (C_p(a_1)+t) \cup (C_p(a_2)-t)\]
is a covering system. 
\end{lem}

\begin{proof}
Let $b \in \Z$. We consider three cases for what covers $b$ in $C$:
\begin{enumerate}
\item $b$ is covered by some arithmetic progression in $C \setminus(C_p(a_1) \cup C_p(a_2))$, 
\item $b$ is covered by some arithmetic progression in $C_p(a_1)$,
\item $b$ is covered by some arithmetic progression in $C_p(a_2)$.
\end{enumerate}

In case (a), it is clear that $b$ is covered in $C'$. Suppose now that we are in case (b). We look at what covers $b+t$ in $C$. If $b+t$ is covered by some arithmetic progression in $C \setminus(C_p(a_1) \cup C_p(a_2))$, then there is some $a \mod m \in C \setminus(C_p(a_1) \cup C_p(a_2))$ such that $b+t \equiv a \mod m$. Note that $p \nmid m$, as or else we would have $a \mod m \in C_p(a_2)$. This implies that $b+t \equiv b \mod m$, and that $b$ is covered in $C'$. We may now suppose that $b+t$ is covered by some arithmetic progression in $C_p(a_2)$, so that there exists some $a \mod m \in C_p(a_2)$ such that $b+t \equiv a \mod m$. It follows that $b \equiv a-t \mod m$, and that $b$ is covered by some arithmetic progression in $C_p(a_2)-t$, so that $b$ is covered in $C'$. 

Case (c) follows in a similar fashion as case (b), and so we omit the details. This completes the proof of the lemma. 
\end{proof}

The next lemma we look at is a rather famous theorem of Mirsky and Newman, see \cite{Mirsky} for a proof. 

\begin{lem}[Mirsky and Newman]\label{Mirsky}
Suppose 
\[C=\{a_1 \pmod{m_1}, \ldots, a_n \pmod{m_n}\}\]
is a distinct covering system. Then
\[\sum_{i=1}^n \frac{1}{m_i}>1.\]
\end{lem}
Finally, we look at a lemma that will be useful in the proofs of Theorem~\ref{51440} and Theorem~\ref{65040}. This lemma is similar to Theorem 1 in \cite{Simpson}, and we defer to there for the proof. 

\begin{lem}\label{smallerp}
Let $m$ be a positive integer. Let $p$ be a prime $\geq m$. Let $L$ be a positive integer, and suppose that $p \nmid L$. Let $q>p$ be the largest prime dividing $L$. If there exists a distinct covering system using all divisors of $L$ that are $\geq m$, then there exists a distinct covering system using all divisors of $p^{\nu_q(L)}L/q^{\nu_q(L)}$ that are $\geq m$. 
\end{lem}

\section{Integer programming and covering systems}\label{integerprogramming}

In this section, we look at a link between integer programming and covering systems, following work of McNew and Setty \cite{McNew}.

Let $\CM$ be a multiset of positive integers. Consider the following question: Does there exist a covering system for which the multiset of moduli is precisely $\CM$, or a subset thereof? We refer to questions of this type as \textit{integer covering} problems. The problems considered by Krukenberg and Dalton and Trifonov, as well as Theorem~\ref{5108} and Theorem~\ref{51440}, may be viewed as integer covering problems.

McNew and Setty \cite{McNew}, in Appendix B of their paper, note that such problems may be viewed as integer programming problems. We give a similar explanation here. 

Let 
\[\CM=\{d_1,d_2,\ldots,d_k\},\]
be a multiset of positive integers. Let
\[M=\{m_1,\ldots,m_n\}\]
be the set of distinct integers in $\CM$, and suppose $m_i$ appears exactly $f_i$ times in $\CM$. Define binary integer variables $x_{i,j}$, where $i$ ranges from $1$ to $n$, and for each $i$, $j$ ranges from $1$ to $m_i$. You should think of $x_{i,j}$ as being $1$ if and only if $j \pmod{m_i}$ is in the covering system we are attempting to construct with multiset of moduli $\CM$. 

With these variables, we need to add constraints to ensure that
\begin{enumerate}
\item the modulus $m_i$ is used at most $f_i$ times, and
\item every integer in $[1,\lcm(m_1,\ldots,m_n)]$ is covered.
\end{enumerate}
To make sure a potential covering system satisfies (a), we add for each $i \in [n]$ the constraint
\[\sum_{j=1}^{m_i} x_{i,j} \leq f_i.\]
To make sure a potential covering system satisfies (b), we add for each $b \in [1,\lcm(m_1,\ldots,m_n)]$ the constraint
\[\sum_{i=1}^n x_{i,b \mod{m_i}} \geq 1,\]
where $b \pmod{m_i}$ is understood here to be the least integer $c \in \{1,\ldots,m_i\}$ such that $c \equiv b \pmod{m_i}$. 

With this set-up, there is a covering for which the multiset of moduli is $\CM$ if and only if there is a solution to the above integer programming problem with $m_1+\cdots+m_n$ variables and $n+\lcm(m_1,\ldots,m_n)$ constraints.

In practice, when setting up an integer programming problem to solve an integer covering problem, we will usually make a few simplifications. First off, using Lemma~\ref{translation} and Lemma~\ref{pauto}, we can and will assume that a few arithmetic progressions are fixed. Let us look at a simple example to illustrate how. 

Suppose we have the integer covering problem with $\CM=\{2,3,4,6,12\}$. By Lemma~\ref{translation}, if there exists a covering system $C$ with these moduli, then there in fact exists a covering system $C'$ with these moduli that contains $0 \mod 2$ and $0 \mod 3$. Indeed, by the Chinese Remainder Theorem, it must be the case that one of $C, C+1, C+2,\ldots,C+5$ contains both $0 \mod 2$ and $0 \mod 3$. Now, if there exists a covering system with these moduli, then there is one in which the arithmetic progression with modulus 6 does not intersect with $0 \mod 2$ and $0 \mod 3$. By applying Lemma~\ref{pauto} with $p=3$, we may assume the arithmetic progression intersects with $1 \mod 3$, and so we may suppose that we in fact have $1 \mod 6$. 

By doing this, we reduce the number of variables in our integer programming model by $2+3+6=11$. We may also get rid of a few constraints. Indeed, every integer in $\{0,1,2,3,4,6,7,8,9,10\}$ is covered by either $0 \mod 2$, $0 \mod 3$, or $1 \mod 6$, and so we do not need the constraints that guarantee that these integers are covered anymore. For this case, we are left with $4+12$ variables coming from the moduli $4$ and $12$, 2 constraints coming from (a), and 2 constraints coming from (b). The problem is thus simplified, and made easier to solve. In this case, of course, the problem is easy to solve, as illustrated by the covering system in the introduction, but the integer covering problems that arise in later sections are not so easy. 
\begin{rem*}
Through private communications with Nathan McNew, it became apparent that they were making similar reductions in \cite{McNew}, but they make no mention of these in their paper.   
\end{rem*}

In the following sections, we will be taking integer covering problems, making the necessary reductions, and then posing the reduced problems as integer programming problems. We then use Gurobi \cite{Gurobi} to solve these problems. If the MIP (Mixed-Integer Programming) solver returns no solutions, then we know that the correponding integer covering problem is infeasible. The relevant code can be found at the following link: https://github.com/JonahKleinGit/Integer-covering-problems

\section{Proof of Theorem~\ref{5108}}\label{5108proof}

In this section, we prove Theorem~\ref{5108}. Suppose there exists a covering system with distinct moduli, all of which are in the interval $[5,107]$. By Lemma~\ref{Krukenberglemma1}, we may suppose that all moduli are divisors of
$2^5 \cdot 3^2 \cdot 5 \cdot 7$. This leaves us with the set of potential moduli
\begin{align*}\CM:=&\{5,6,7,8,9,10,12,14,15,16,18,20,21,24,28,30,32,35,36,40,42,45,48,56,60,63,70,\\
&72,80,84,90,96,105\}.
\end{align*}
By Lemma~\ref{Krukenberglemma2}, we may replace $\{32,96\}$ by a single arithmetic progression with modulus $48$, leaving us with
\begin{align*}\CM:=&\{5,6,7,8,9,10,12,14,15,16,18,20,21,24,28,30,35,36,40,42,45,48,48,56,60,63,70,\\
&72,80,84,90,105\}
\end{align*}
as our potential moduli. By Lemma~\ref{translation}, if a covering system exists with these moduli, then one exists which contains $\{4 \mod 5, 6 \mod 7, 7 \mod 8, 8 \mod 9\}$. If such a covering system exists with the moduli set above, we may additionally assume that our arithmetic progression with modulus 35 does not intersect with both $4  \mod 5$ and $6 \mod 7$. By applying Lemma~\ref{pauto} successively with $p=5$ and $p=7$, we may suppose that such a covering would contain $33 \mod{35}$. 

We then set up this problem as an integer programming problem using Gurobi. After about 10 hours, the model was deemed infeasible, which implies that there is no covering system with multiset of moduli $\CM$. This proves Theorem~\ref{5108}. 
\section{Proof of Theorem~\ref{51440}}\label{51440proof}
In this section, we prove Theorem~\ref{51440}. To prove this theorem, we need to show that any positive integer $<1440$ cannot be the least common multiple of the moduli in a distinct covering system with minimum modulus 5. Let $0<L<1440$ be a candidate for such an integer, where necessarily 5 divides $L$. If there exists a distinct covering system $C$ with minimum modulus 5 and least common multiple $L$, then we may add in any missing divisors of $L$ that are $\geq 5$ as moduli in this covering system and still have a covering system. Thus, we may suppose that all divisors of $L$ that are $\geq 5$ are moduli in $C$. 

Lemma~\ref{Mirsky} implies that
\[\sum_{\substack{d|L\\d \geq 5}}\frac{1}{d} >1.\]
This leaves us with the list of candidates
\[\{240,360,420,480,540,600,630,720,840,900,960,990,1050,1080,1200,1260,1320\}.\]
We may eliminate all candidates that are $<720$ by noticing that if there is a distinct covering system with divisors of $L_1$ that are $ \geq 5$, and $L_1|L_2$, then there is also a distinct covering system with divisors of $L_2$. We can also eliminate 990 and 1320 by Lemma~\ref{smallerp}. Thus, we are left with the set of 8 candidates
\[K=\{720,840,900,960,1050,1080,1200,1260\}.\]
For each $L \in K$, we are left with an integer covering problem in which 
\[\CM=\{d\geq 5: d|L\}.\]
We now look at how to deal with $L=720$. Our potential moduli are all divisors of $720$ that are $\geq 5$, that is
\[\CM=\{5,6,8,9,10,12,15,16,18,20,24,30,36,40,45,48,60,72,80,90,120,144,180,240,360\}.\]
By Lemma~\ref{translation}, we may suppose that if a covering system exists with these moduli, then one exists that contains $\{4 \mod 5, 7 \mod 8,8 \mod9\}$. We convert this reduced problem as an integer programming problem, and ask the solver to solve it for us. The solver deems that the model is infeasible within a few seconds, and so there is no covering system with moduli set $\CM$. The other 7 cases are treated in similar ways, and are all deemed infeasible in a few seconds. This proves Theorem~\ref{51440}. 

\section{Proof of Theorem~\ref{65040}}\label{65040proof}
In this section, we prove Theorem~\ref{65040}. We proceed in the same fashion as the proof of Theorem~\ref{51440}. First off, Lemma~\ref{Mirsky} implies that the possible candidates are
\begin{align*}&\{504, 720, 840, 1008, 1080, 1260, 1440, 1512, 1680, 1800, 1848, 1890, 
1980, 2016, 2100, 2160, 2520,\\
& 2640, 2772, 2880, 3024, 3120, 3150, 
3168, 3240, 3276, 3360, 3528, 3600, 3696, 3780, 3960, 4032,\\
& 4200, 
4320, 4368, 4536, 4620, 4680, 4752\}\end{align*}
We may eliminate all candidates that are $<2520$. Using Lemma~\ref{smallerp}, we may also eliminate all candidates in the set
\[\{2640,3120,3168,3276,3960,4368,4680,4752\},\]
 leaving us with the 16 candidates in
\begin{align*}\{2520, 2772, 2880, 3024, 3150, 3240, 3360, 3600, 3528, 3696,3780, 4032,4200, 4320,4536,4620\}.\end{align*}
We eliminate each one at a time by reducing and then solving them as integer programming problems. Most are done rather quickly. Only $2520$ and $4320$ take a significant amount of time to solve, with $2520$ taking the longest time to solve, at around 2.5 hours. This proves Theorem~\ref{65040}. 

\section{Proof of Theorem~\ref{6192}}\label{6192proof}

The covering system
\begin{align*}
\{& 2 \mod 6, 6 \mod 7, 7 \mod 8, 6 \mod 9, 4 \mod {10}, 10 \mod{11}, 5 \mod{12}, 5 \mod{14},\\
& 10 \mod{15}, 11 \mod{16}, 12 \mod{18}, 8 \mod{20}, 10\mod{21},9 \mod{22}, 22\mod{24},\\
& 23 \mod{25}, 18\mod{27}, 25\mod{28}, 16 \mod{30}, 19 \mod{32}, 6 \mod{33}, 29 \mod{35},\\
& 21 \mod{36}, 18 \mod{40}, 37 \mod{42}, 41 \mod{44}, 0 \mod{45}, 19 \mod{48}, 18 \mod{50},\\
& 36 \mod{54}, 47 \mod{55}, 21 \mod{56}, 52 \mod{60}, 45 \mod{63}, 27 \mod{66},43 \mod{70}, \\
&3 \mod{72}, 63 \mod{75}, 0 \mod{77}, 49 \mod{80}, 1 \mod{84}, 73 \mod{88}, 72 \mod{90}, 35 \\
&\mod{96}, 58 \mod{100}, 57 \mod{105}, 9 \mod{108}, 57 \mod{110}, 105 \mod{112}, \\
&82 \mod{120}, 9 \mod{126}, 81 \mod{132}, 81 \mod{135}, 133 \mod{140}, 99 \mod{144}, \\
&78 \mod{150},133 \mod{154}, 49 \mod{168}\}
\end{align*}
demonstrates a proof of Theorem \ref{6192}. 
\begin{rem*}
The above covering system was found by Ognian Trifonov. 
\end{rem*}
\section*{Acknowledgements}
The author would like to thank Michael Filaseta and Ognian Trifonov for their advice, and Nathan McNew for a helpful conversation and for sending his preprint. 

\section*{Funding}
The author is supported by a scholarship from the Natural Sciences and Engineering Research Council of Canada (NSERC).

\bibliographystyle{plain}

\end{document}